\documentclass[12pt]{article}
\usepackage{graphicx}
\usepackage{amsmath}
\usepackage{amsfonts}
\usepackage{amssymb}
\usepackage{color}
\usepackage{float}
\providecommand{\U}[1]{\protect\rule{.1in}{.1in}}
\newtheorem{theorem}{Theorem}

\newenvironment{proof}[1][Proof]{\noindent\textbf{#1.} }{\ \rule{0.5em}{0.5em}}

\begin{document}

$\ $

\vspace{2.cm}

\begin{center}
{\Large \textbf{Platonic Compounds of Cylinders}}

\vspace{.4cm} {\large \textbf{Oleg Ogievetsky$^{1,2,3}$
and Senya Shlosman$^{1,4,5}$}}

\vskip .3cm $^{1}$Aix Marseille Universit\'{e}, Universit\'{e} de
Toulon, CNRS, \\ CPT UMR 7332, 13288, Marseille, France

\vskip .05cm $^{2}$I.E.Tamm Department of Theoretical Physics, 
Lebedev Physical Institute,
Leninsky prospect 53, 119991, Moscow, Russia

\vskip .05cm $^{3}$Kazan Federal University, Kremlevskaya 17, Kazan
420008, Russia

\vskip .05cm $^{4}$Inst. of the Information Transmission Problems, RAS,
Moscow, Russia

\vskip .05cm $^{5}$Skolkovo Institute of Science and Technology,
Moscow, Russia

\end{center}

\vskip .4cm
\hfill{\sf 
... and the left leg merely does a forward aerial half turn}

\hspace{3.86cm}{\sf  every alternate step.}

\vskip .3cm
\hfill {\it Monty Python} $\ \ \ \ \ \ \ \ $

$\ $

\begin{abstract}\noindent
Generalizing the octahedral configuration of six congruent cylinders touching the unit sphere, we exhibit configurations of congruent cylinders associated 
to a pair of dual Platonic bodies. 
\end{abstract}

\newpage
\tableofcontents

\newpage
\section{Introduction}

In this paper we study \textit{compounds} of cylinders. By a compound $\varkappa$ of cylinders we mean a configuration of infinite open
nonintersecting right circular congruent cylinders in $\mathbb{R}^{3},$ touching the unit sphere $\mathbb{S}^{2}$, such that some of the cylinders in $\varkappa$ touch each other. 
We denote by $r\left(  \varkappa\right) $ the common radius of cylinders forming the compound. 

\vskip .2cm
Let $M^{k}$ be the manifold of $k$-tuples of straight lines
$$\mathbf{m}=\left\{  u_{1},...,u_{k}\right\}  $$ tangent to the sphere $\mathbb{S}^{2}$. The manifold $M^{k}$ has dimension $3k.$
The space of compounds of $k$ cylinders can be conveniently identified with the manifold $M^{k}$: the straight line corresponding to a cylinder is its unique ruling touching the unit sphere. Therefore the space of compounds of $k$ cylinders is $3k$-dimensional.

\vskip .2cm
For a configuration $\mathbf{m}$ we denote by $d\left(  \mathbf{m}\right)$ the minimal distance between the lines of the configuration.

\vskip .2cm
If $\varkappa$ is a compound with tangent rulings $\mathbf{m}$, then
\begin{equation}\label{radidi}
r\left(  \varkappa\right)  =\frac{d\left(  \mathbf{m}\right)  }{2-d\left( \mathbf{m}\right)  }\ .
\end{equation}
By a slight abuse of notation we allow the lines $u_{i}$ to intersect; then $r\left(  \varkappa\right)  =d\left( \mathbf{m}\right)  =0.$

\vskip .2cm
For every $k\geq3$ let us define the number
\[ R_{k}=\max_{\varkappa\in M^{k}}r\left(  \varkappa\right)\  .\]
By definition, $R_{k}$ is the maximal possible radius of $k$ congruent nonintersecting cylinders touching $\mathbb{S}^{2}.$ It is easy to see that
$R_{3}=3+2\sqrt{3}$, and a natural (open) conjecture is that $R_{4}=1+\sqrt{2}.$

\vskip .2cm
The compound $C_{6}$ of six unit parallel cylinders has appeared in one question of Kuperberg \cite{K} in 1990. For quite some time the common belief was that
$R_{6}=r\left(  C_{6}\right)  =1,$ until M. Firsching \cite{F} in 2016 has found by numerical exploration of the manifold $M^{6}$ a compound $\varkappa_{F}\in
M^{6}$ with $r\left(  \varkappa_{F}\right)  \approx1.049659.$ That finding has triggered our interest in the problem, and in the paper [OS] we have
discovered a curve $C_{6}(t)$ in the compound space $M^{6}$, $C_{6}\left(  0\right)  =$ $C_{6},$ along which the function $r\left(
C_{6}\left(  t\right)  \right)  $ increases to its maximal value on this curve, $r_{\mathfrak{m}}=\frac{1}{8}\left(  3+\sqrt{33}\right)  \approx
1.093070331,$ taken at the compound $C_{\mathfrak{m}}.$ 

\vskip .2cm
The compounds $C_6$ and $C_{\mathfrak{m}}$  are shown on Figure \ref{confCcyl} (the green unit ball is in the center).
\begin{figure}[H]
\vspace{-1cm}
\centering
$\ \ $\raisebox{.3cm}{\includegraphics[scale=0.745]{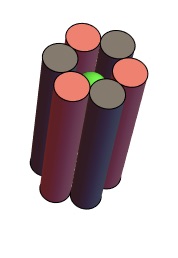}}$\ \ $\includegraphics[scale=0.36]{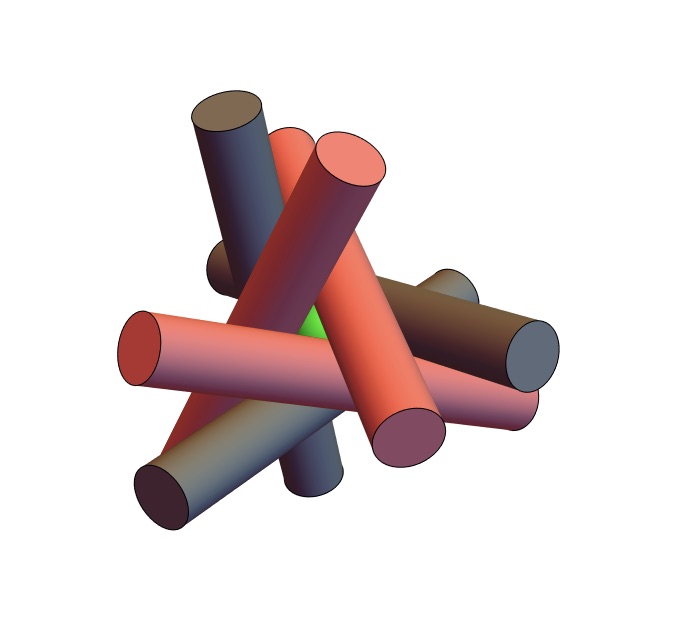}
\parbox{11.8cm}{
\caption{Two compounds of cylinders: the compound $C_6$ of six parallel cylinders of radius 1 (on the left) and the compound $C_{\mathfrak{m}}$ of six cylinders of radius $\,\approx\! 1.0931$ (on the right) } }
\label{confCcyl}
\end{figure}

\vskip .2cm
We have shown in [OS-C6] that in fact the compound $C_{\mathfrak{m}}$ is a point of local maximum (modulo global rotations) of
the function $r:$ for every compound $\varkappa$ close to $C_{\mathfrak{m}}$ we have $r\left(  \varkappa\right)  <r_{\mathfrak{m}}.$ 

\vskip .2cm
We believe that in fact
\[ R_{6}=r_{\mathfrak{m}}=\frac{1}{8}\left(  3+\sqrt{33}\right)  .\]
For values of $k$ larger than 3 the quantities $R_{k}$ are {unknown; at the end of this paper we present some lower bounds for few values of $k$.

\vskip .2cm
In his paper Kuperberg has presented yet another compound, $O_{6}\in M^{6}$ with the same radius $r\left(  O_{6}\right)  =1$. 

\vskip .2cm
The compound $O_6$ is shown on Figure \ref{octahedrConfc}; the central unit ball is again shown in green. 
\begin{figure}[H]
\vspace{-.42cm} \centering
\includegraphics[scale=0.52]{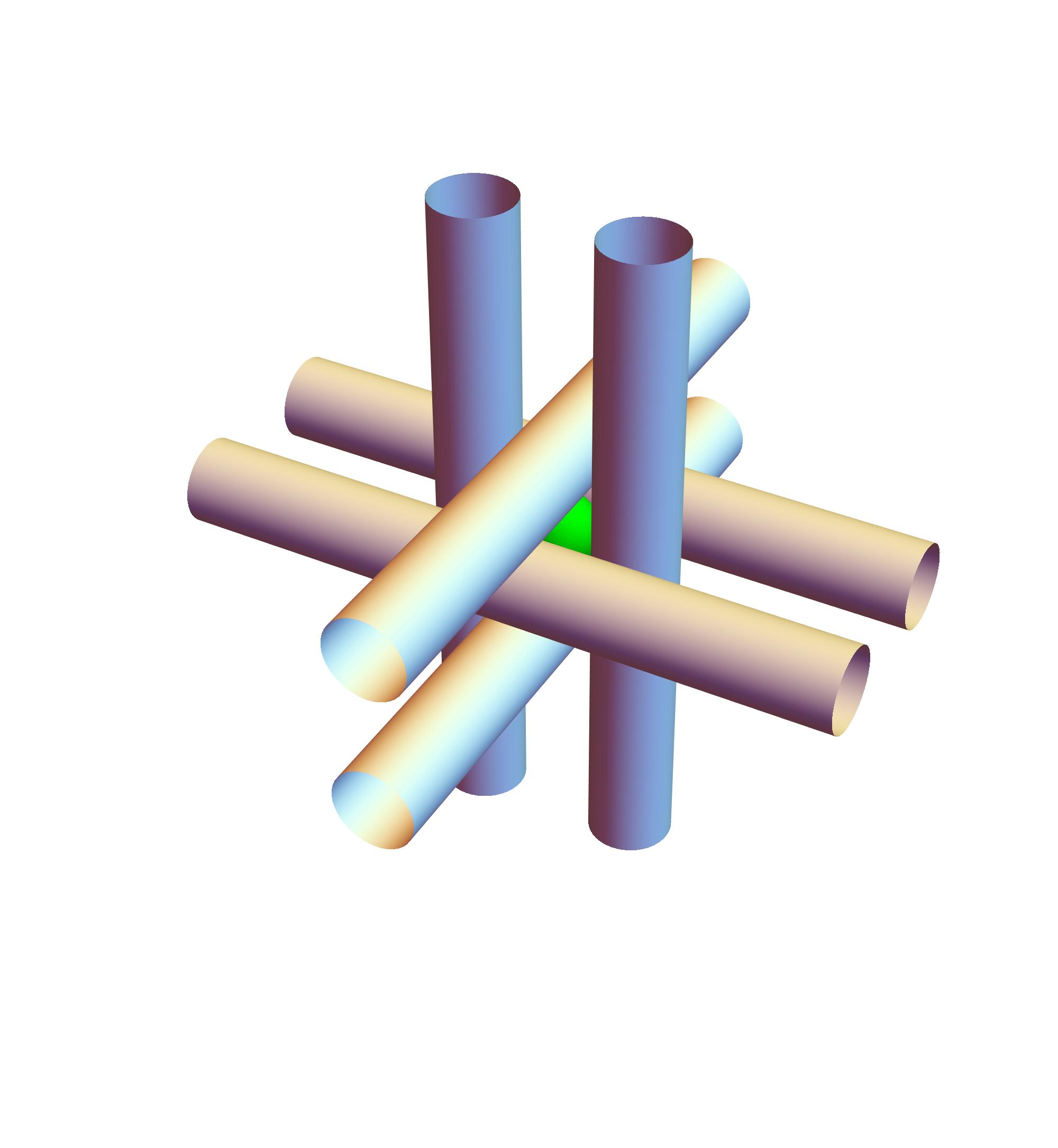}
\vspace{-1cm}
\caption{Compound $O_6$}
\label{octahedrConfc}
\end{figure}

Contrary to $C_{6}$, the compound $O_{6}$ is indeed a point of local maximum (modulo global rotations) of the function $r:$ for
every compound $\varkappa$ close to $O_{6}$ one has $r\left(  \varkappa \right)  <1.$ This was proved in our paper [OS-O6]. 

\vskip .2cm
The configuration $O_6$ of tangent lines is shown on Figure \ref{octahedrConft}.
\begin{figure}[H]
\vspace{-.3cm} \centering
\includegraphics[scale=0.78]{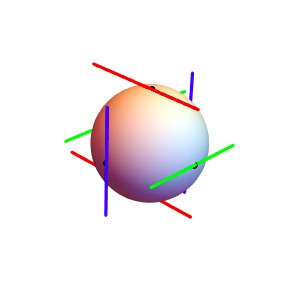}
\vspace{-1.4cm}
\caption{Configuration $O_6$ of tangent lines}
\label{octahedrConft}
\end{figure}

Sometimes the configuration $O_6$ is called `octahedral' because, probably, the tangency points lie at the vertices of the regular octahedron.
We present an interpretation of the configuration $O_6$ which relates it to the configuration of rotated 
edges of the regular tetrahedron. Following this interpretation we introduce configurations of tangent lines  
for each pair of dual Platonic bodies, that is, for the tetrahedron/tetrahedron, octahedron/cube and icosahedron/dodecahedron; 
for the pair tetrahedron/tetrahedron this is precisely the locally maximal configuration $O_6$.
The number of lines in our configurations is equal to the number of edges of either of Platonic bodies in the pair, that is, 
twelve for the pair octahedron/cube and thirty for the pair icosahedron/dodecahedron.

\vskip .2cm
We recall the definition of a \textit{critical configurations} from \cite{OS-O6}. We call a cylinder configuration \textit{critical} if for each small deformation $t$ 
that keeps the radii of the cylinders, either
\begin{itemize}
\item[(T1)] some cylinders start to intersect, or else
\item[(T2)] the distances between all of them increase, but by no more than
$$\sim\left\Vert t\right\Vert ^{2}\ ,$$
or stay zero, for some. 
\end{itemize}
Here $\left\Vert t\right\Vert $ is the natural norm, see details in \cite{OS-O6}.

\vskip .2cm
\noindent A critical configuration is called a \textsl{locally maximal} configuration if all its deformations are of the first type. Any other critical configuration  is called a
\textsl{saddle configuration}, and the deformations of type (T2) are then called the \textsl{unlocking} deformations.

\vskip .2cm
For the pair octahedron/cube we find two possible candidates for the critical points; For the pair  icosahedron/dodecahedron we find four possible candidates for the critical points.

\vskip .3cm
The paper is organized as follows. In Section \ref{O6tetrahedron} we present an interpretation of the configuration $O_6$ relating it to the edges of the tetrahedron. In Sections \ref{confplatonic}, \ref{placo} and \ref{mincom} we describe configurations of tangent lines associated to a pair of dual Platonic bodies.
We conclude the article by several conjectures 
concerning local maxima and critical points for the configurations related to the pairs octahedron/cube and icosahedron/dodecahedron. 

\section{Interpretation of the configuration $O_6$:\\ $\mathbb{A}_4$-symmetric configurations\vspace{.2cm}}\label{O6tetrahedron}
We shall now give an interpretation of the configuration $O_{6}$ which shows that it rather deserves
to be called the tetrahedral configuration.

\vskip .2cm
There is a family of configurations possessing the $\mathbb{A}_4$ symmetry ($\mathbb{A}_4$ is the group of proper symmetries of the regular tetrahedron,
alternating group on four letters). We start with the configuration of the tangent to the unit sphere lines which are continuations of the edges of the regular tetrahedron.
The points of the sphere at which tangent lines pass are the
edge middles of the regular tetrahedron. The initial position of the edges of the tetrahedron are shown in blue on Figure \ref{A4configurations}.

\begin{figure}[H]
\vspace{.2cm} \centering
\includegraphics[scale=0.5]{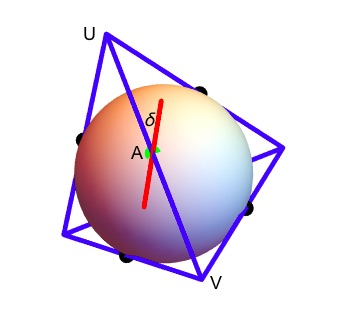}
\vskip -.2cm
\caption{Sphere tangent to tetrahedron edges}
\label{A4configurations}
\end{figure}

Then each edge is rotated around the diameter of the unit sphere, passing through the middle of the edge, by an angle $\delta$. On Figure \ref{A4configurations} the point $A$ (in green) is the middle of the edge $UV$. The line, passing through the point A and rotated by the angle $\delta$, is shown in red. The lines passing through
other middles of edges are rotated according to $\mathbb{A}_4$ symmetry. The positions of other lines are well defined because the stability subgroup of the edge $UV$ of the regular tetrahedron is generated by the rotation by angle $\pi$ around the diameter of the sphere passing through the point $A$ and, clearly, the red line on Figure \ref{A4configurations} is invariant under this rotation.

\vskip .2cm
The minimum of the squares of the distances between the lines is given by the formula
\begin{equation}\label{forTT}d^2=\frac{16 T^2}{(3T^2+1)(T^2+3)}\equiv -\,\frac{4\sin^2 (2\delta)}{\bigl(\cos (2\delta)-2\bigr)\,\bigl(\cos (2\delta)+2\bigr)}\ .\end{equation}
Here $T=\tan (\delta)$. When $\delta$ changes from 0 to $\pi/2$, the value  of $d^2$ increases, achieves the maximal value 1 for $\delta_0=\pi/4$, and afterwards decreases to 0, see Figure \ref{DistanceGraphA4}.
\begin{figure}[H]
\vspace{.2cm} \centering
\includegraphics[scale=0.6]{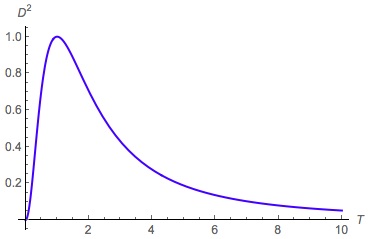}
\vskip -.2cm
\caption{Graph of $d^2(T)$ }
\label{DistanceGraphA4}
\end{figure}

For $\delta\neq 0,\pi/4,\pi/2$ the symmetry of the configuration is $\mathbb{A}_4$. For
$\delta=0$ the symmetry group is $S_4$, the full symmetry group of the regular tetrahedron. For $\delta=\pi/4$ the configuration is centrally symmetric, its  symmetry group is $\mathbb{A}_4\times C_2$, where the cyclic group $C_2$ is
generated by the central reflection $\mathcal{I}$. For $\delta=\pi/4$ this is exactly the configuration $O_6$.
For $\delta=\pi/2$ the tangent lines become again the continuations of the edges of the regular tetrahedron.

\vskip .2cm
The convex span of the edge middles of the regular tetrahedron is the regular octahedron. Our above observation boils down to the fact that rotating the 
edges in the plane tangent to the mid-sphere (the sphere which touches the middle of each edge of $\mathcal{P}$) we obtain the configuration $O_6$ when the 
angle of rotation is equal to $\pi/4$.

\vskip .2cm
The following pictures of models made with the help of a tensegrity kit (Figure \ref{TensegrityModels}) and its
Martian version (Figure \ref{MartianVersion}) might be helpful for understanding.
\begin{figure}[H]
\centering
\includegraphics[scale=0.8]{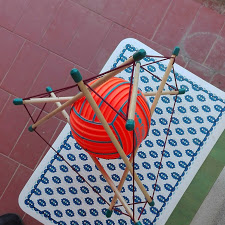}\hspace{.4cm}\includegraphics[scale=0.8]{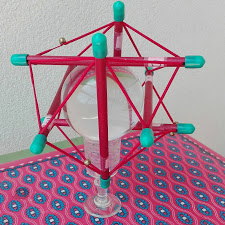}
\caption{Tensegrity models. The rotating edges of the tetrahedron (left). The position of edges at angle $\pi/4$ is the configuration $O_6$ (right). }
\label{TensegrityModels}
\end{figure}
\begin{figure}[H]
\vspace{-.5cm} 
\centering
\includegraphics[scale=0.22]{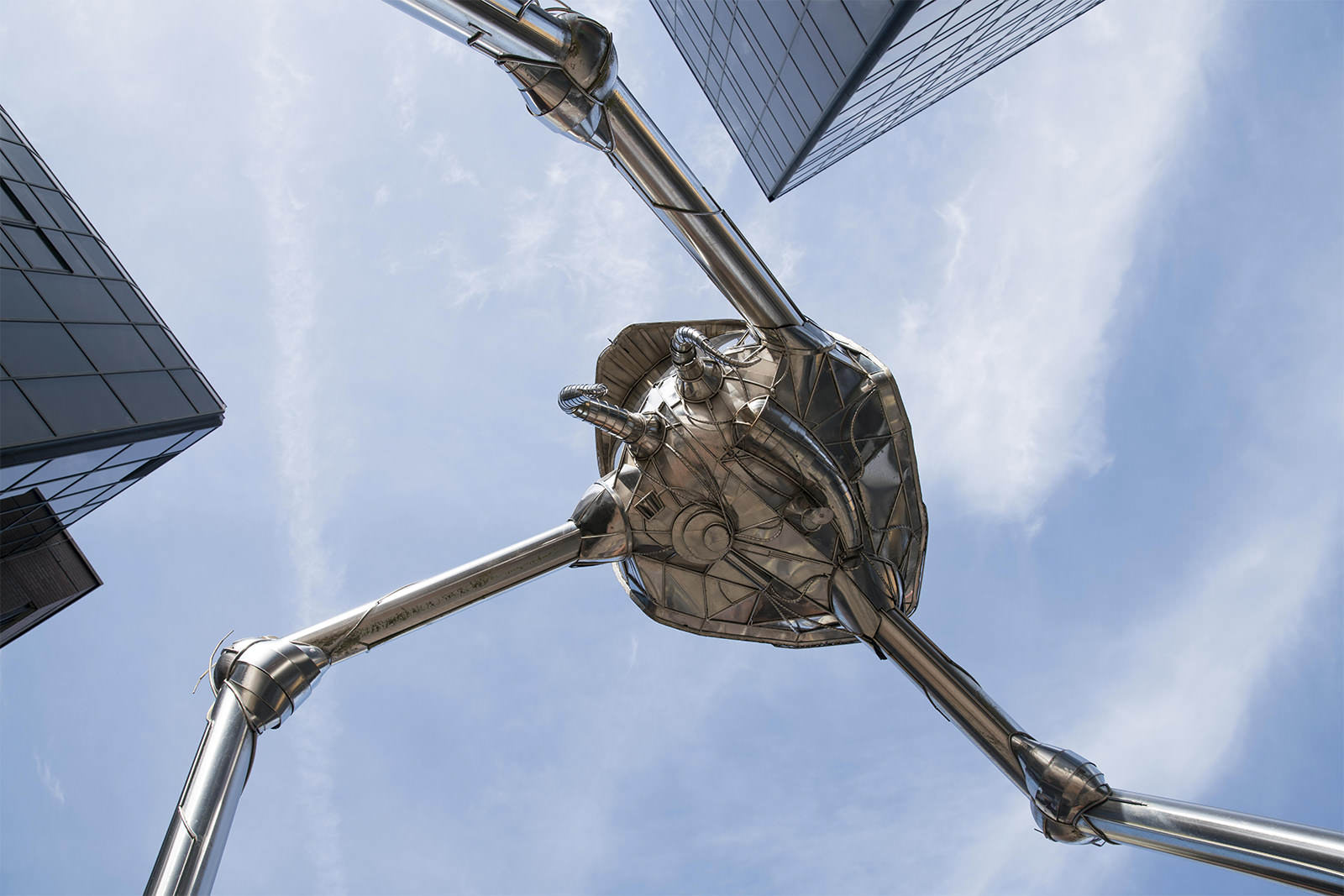}
\vskip -.2cm
\caption{Martian version (located in Woking, the hometown of H. G. Wells) of the rotating upper three edges of the tetrahedron, view from the center of the unit sphere.}
\label{MartianVersion}
\end{figure}

\section{Generalizations of the configuration $O_6$:\\ $\delta$-rotation process\vspace{.2cm}}\label{confplatonic}
We denote by $\mathcal{T}$, $\mathcal{O}$ and $\mathcal{I}$, respectively, the dual pairs tetrahedron/tetra\-hedron, octahedron/cube and icosahedron/dodecahedron of Platonic bodies. 
Let $N_\mathcal{T}=6$, $N_\mathcal{O}=12$ and $N_\mathcal{I}=30$ be the number of edges of a Platonic body in the corresponding pair.
Finally let $G_\mathcal{O}\cong\mathbb{S}_4$ and $G_\mathcal{I}\cong\mathbb{A}_5$ be the proper symmetry groups of the dual pairs octahedron/cube and icosahedron/dodecahedron.

\vskip .2cm 
A construction, similar to the one described in Section \ref{O6tetrahedron}, works for 
each pair of dual Platonic bodies. Namely, let a unit sphere touch the edge middles
of a Platonic body $\mathcal{P}$. Rotate the edges of $\mathcal{P}$ around the axes passing through the center of the sphere and tangency points. By the same reason as for the tetrahedron, it can be done preserving the proper symmetry group of $\mathcal{P}$; spelled differently, we rotate each edge in the same direction, say, counterclockwise, viewed from the tip of the axis, passing through the center of the sphere and the tangency point. When the angle $\delta$ of the rotation reaches $\pi/2$ the edges form the configuration of edges of the dual to $\mathcal{P}$ Platonic body. 

\vskip .2cm
In the subsequent Sections we investigate the details of this `$\delta$-{\it rotation}' process for the remaining pairs $\mathcal{O}$ and $\mathcal{I}$.

\subsection{Neighboring tangent lines\vspace{.1cm}}\label{neitali}
For a Platonic body $\mathcal{P}$, assume that the symmetry axis $\mathsf{a_v}$ passing through the vertex $\mathsf{v}$ of $\mathcal{P}$ is $\mathsf{k}$-fold. 
By `neighboring' edges of $\mathcal{P}$ we mean edges which share a vertex $\mathsf{v}$ and can be obtained from one another by a rotation 
through the angle $2\pi/\mathsf{k}$ around the axis $\mathsf{a_v}$. We keep the name `neighboring' for the tangent lines obtained from the neighboring edges in the 
$\delta$-rotation process. Note that this notion does not depend on a concrete Platonic body in a pair: neighboring tangent lines for a Platonic body $\mathcal{P}$
remain neighboring for the dual of $\mathcal{P}$.

\vskip .2cm
For some value of $\delta$ the distance between the neighboring tangent lines takes the maximal value.

\vskip .2cm
We now present the maximal values $d_\mathcal{O}$
and $d_\mathcal{I}$  between the neighboring tangent lines for the pairs $\mathcal{O}$ and $\mathcal{I}$. 

\subsubsection{Pair $\mathcal{O}$\vspace{.1cm}}\label{PairO}
For the pair octahedron/cube, let the octahedron correspond to $\delta=0$. 

\begin{theorem} The maximum distance in the $\delta$-rotation process for the pair $\mathcal{O}$ is
achieved at the angle $\delta_\mathcal{O}$ such that
\begin{equation}\label{valdodico0}\tan (\delta_\mathcal{O})=\frac{3^{1/4}}{\sqrt{2}}\end{equation}
(approximately, $\delta_\mathcal{O}\simeq 0.23856 \pi\simeq 0.74946$).
The square of the maximal distance between the tangent lines is the value of the function
\begin{equation}\label{valdodico}-\,\frac{4\sin^2 (2\delta)}{\bigl(\cos (2\delta)-3\bigr)\,\bigl(\cos (2\delta)+5\bigr)}
\end{equation}
at $\delta=\delta_\mathcal{O}$. It is equal to 
\begin{equation}\label{valdodicob}d_\mathcal{O}^2=2-\sqrt{3}=\frac{(-1+\sqrt{3})^2}{2}\simeq 0.26795\ .\end{equation}
\end{theorem}
\begin{proof} The square of the distance between the rotated neighboring lines is given by the formula (\ref{diab}) which reduces to (\ref{valdodico}) for the octahedron. 

\vskip .2cm
A straightforward but more lengthy verification shows that in the configuration obtained by the $\delta$-rotation through the angle $\delta_\mathcal{O}$ all other 
(non-neighboring) pairwise distances between the tangent lines are bigger than $d_\mathcal{O}$ so the other pairs of corresponding cylinders do not intersect. See Section \ref{glopiOC} for details.
\end{proof}

\vskip .2cm
The corresponding radius of touching cylinders, see formula (\ref{radidi}), is 
\begin{equation}\label{mradiusCO}r_\mathcal{O}=\frac{\sqrt{3}-1}{1+2\sqrt{2}-\sqrt{3}}=7-\sqrt{2}-4\sqrt{3}+3\sqrt{6}\approx 0.3492\ .\end{equation}
The resulting compound, at the angle $\delta_\mathcal{O}$, of the twelve cylinders of radius $r_\mathcal{O}$ is shown on Figures \ref{OC-1} and \ref{OC-2}.

\vskip .2cm
\noindent{\bf Remark.} \emph{Let $C_{12}$ be the configuration of twelve vertical equidistant lines tangent to the unit sphere. We denote by the same symbol 
$C_{12}$  the corresponding configuration of touching vertical cylinders. In \cite {OS} we have proved that the configuration $C_{12}$ is unlockable along a curve, 
which keeps the $\mathbb{D}_6$ symmetry of the configuration, in the moduli space. 
It is interesting to know if the maximal point on this curve beats the radius (\ref{mradiusCO}) of cylinders.}

\newpage
\begin{figure}[H]
\vspace{-.2cm} 
\centering
\includegraphics[scale=0.415]{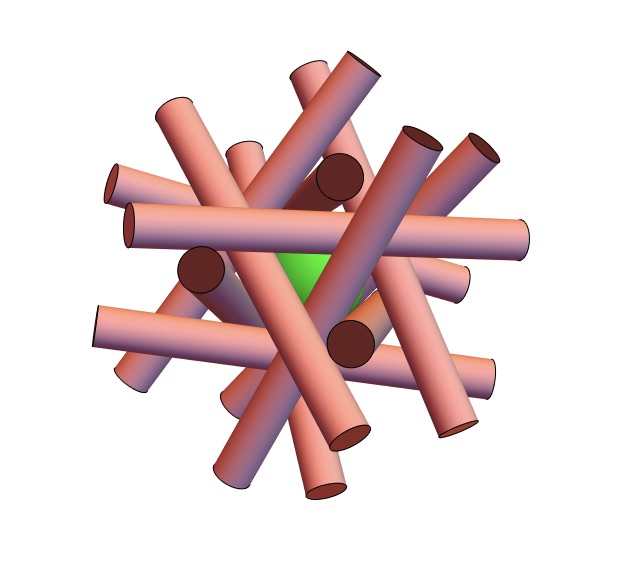}
\caption{Octahedron/cube maximal configuration of cylinders, view from a vertex of the cube}
\label{OC-1}
\end{figure}
\begin{figure}[H]
\centering
\includegraphics[scale=0.395]{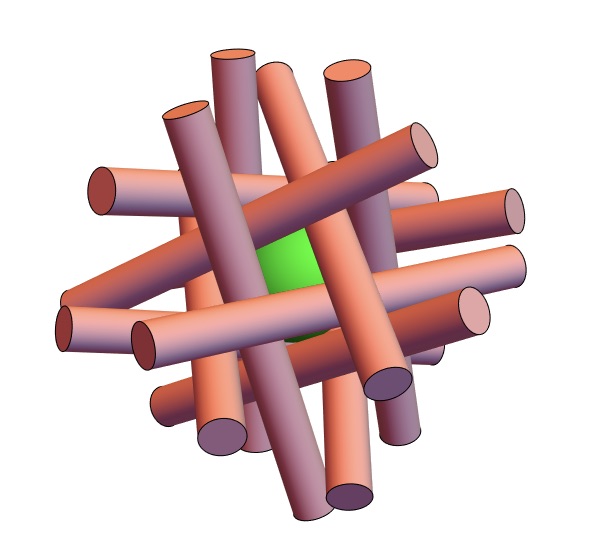}
\caption{Octahedron/cube maximal configuration of cylinders, view from a vertex of the octahedron}
\label{OC-2}
\end{figure}

\filbreak
\subsubsection{Pair $\mathcal{I}$\vspace{.1cm}}\label{PairI}
For the pair icosahedron/dodecahedron, let the icosahedron correspond to $\delta=0$. Then the maximum of the square of
distances between the neighboring edges is achieved at the angle $\delta_{\mathcal{I}}$ such that 
$$\tan (\delta_\mathcal{I})=\left(\frac{6}{5+\sqrt{5}}\right)^{1/4}$$
(approximately, $\delta_\mathcal{I}\simeq
0.24255 \pi$).  The square of the maximal distance between the neighboring tangent lines is the value of the function
\begin{equation}\label{valdodico2}-\,\frac{4\sin^2 (2\delta)}{\left(\cos (2\delta) -(4+\sqrt{5})\right)\cdot
\left(\cos (2\delta)-(1-2\sqrt{5})\tau^3\right)}
\end{equation}
at $\delta=\delta_\mathcal{I}$. Here $\tau$ is the golden ratio,
$$\tau=\frac{1+\sqrt{5}}{2}\ .$$

\noindent{\bf Remark.} \emph{Compare the golden integers, appearing in the denominator of  (\ref{valdodico2}) with the numbers $\beta_{11}$ and
$\gamma_{-19}$, see  Section 4.1 in \cite{OS-C6},  
which enter the formulas for the differentials of distances around the 
configuration $C_{6}\left(  \varphi_{\mathfrak{m}},\delta_{\mathfrak{m}},\varkappa_{\mathfrak{m}}\right)$.
This coincidence hints at some hidden relation between the configuration $C_{6}\left(  \varphi_{\mathfrak{m}},\delta_{\mathfrak{m}},\varkappa_{\mathfrak{m}}\right)$ and the exceptional Platonic bodies (the icosahedron and dodecahedron).}

\vskip .2cm
The value of the function (\ref{valdodico2}) at $\delta=\delta_\mathcal{I}$ is 
\begin{equation}\label{valdodico3}\begin{array}{rcl}d_\mathcal{I}^2&=&
\displaystyle{ \frac{2\sqrt{6\left(5+\sqrt{5}\right)}}{75+33\sqrt{5}+12\sqrt{6\left(5+\sqrt{5}\right)}+5\sqrt{30\left(5+\sqrt{5}\right)}} }\\[2.4em]
&=&
\displaystyle{ \frac{9-\sqrt{5}-\sqrt{6\left(5+\sqrt{5}\right)}}{4}=\left(
\frac{5^{1/4}\sqrt{3\tau}}{2\tau}-\frac{\tau}{2}
\right)^2\simeq 0.0437
\ .}\end{array}
\end{equation}
The corresponding radius of touching neighboring cylinders is 
\begin{equation}\label{valdodico4}r_\mathcal{I}=11-5\sqrt{5}+\sqrt{3\left( 85-38\sqrt{5}\right)}\approx 0.1167\ .\end{equation}
This compound of cylinders is shown on Figure \ref{IntersectingCylindersIO}. 

\vskip .2cm
In contrast to the pair $\mathcal{O}$, there are non-neighboring tangent lines 
at a smaller (than for the neighboring tangent lines) distance; see the overlapping blue cylinders on Figure \ref{IntersectingCylindersIO}. 
The final result concerning the maximal point for the pair $\mathcal{I}$ is contained in Subsection \ref{glpicID}.
\begin{figure}[H]
\vspace{.2cm} 
\centering
\includegraphics[scale=0.46]{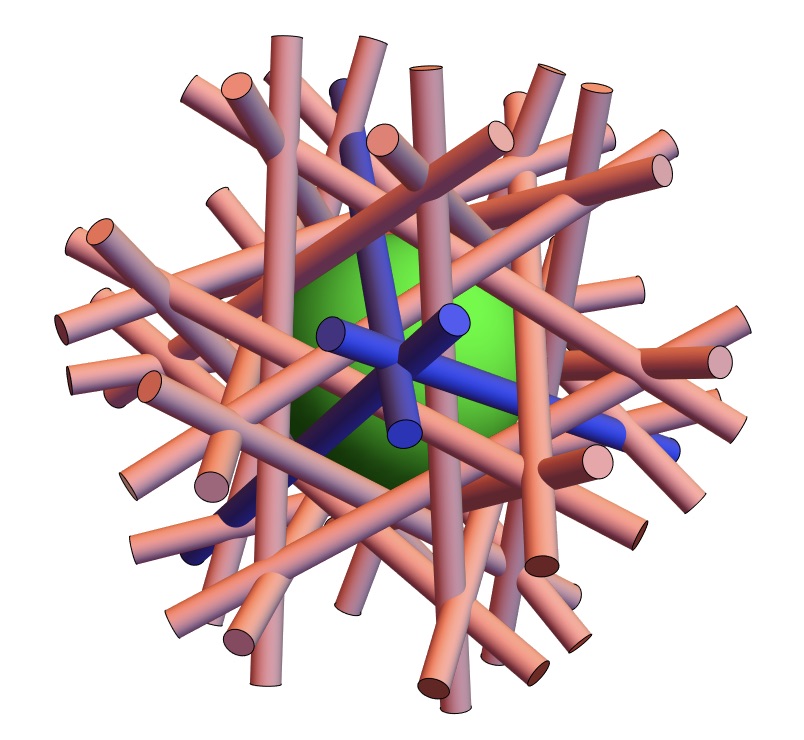}
\caption{Configuration of cylinders of radius $r_\mathcal{I}$ at $\delta=\delta_\mathcal{I}$}
\label{IntersectingCylindersIO}
\end{figure}

\vskip .2cm
The square of the minimal distance at the angle $\delta_{\mathcal{I}}$ is approximately $0.00291762$ which corresponds to a much smaller cylinder radius, equal approximately to
$0.0277571$, compare with (\ref{valdodico3}) and (\ref{valdodico4}). 

\vskip .2cm
A more detailed analysis, given in Subsection \ref{glpicID}, of the $\delta$-rotation process, shows that the actual radius 
of cylinders at the maximal point is not so much smaller than the one given by the expression (\ref{valdodico4}); the maximal
radius is approximately $0.115558$.

\paragraph{Remark.} It is interesting to note that the values $d_\mathcal{I}^2$ and $r_\mathcal{I}$, see formulas (\ref{valdodico3}) and (\ref{valdodico4}), define the same extension of the field of rational numbers, $\mathbb{Q}[d_\mathcal{I}^2]\cong \mathbb{Q}[r_\mathcal{I}]$. Indeed, 
$$r_\mathcal{I}=12 (d_\mathcal{I}^2)^3-62(d_\mathcal{I}^2)^2+74d_\mathcal{I}^2-3\ .$$
It is not so for the pair $\mathcal{O}$, see formulas (\ref{valdodicob}) and (\ref{mradiusCO}).

\subsection{Comment about calculations in Subsection \ref{neitali}\vspace{.1cm}}
To find the values $\delta_\mathcal{O}$ and $\delta_\mathcal{I}$, we place a vertex of a Platonic body $\mathcal{P}$ at the
point $(0,0,h)$, $h>1$, on the $\mathbf{z}$-axis. We draw the line $\iota$ in the plane $\mathbf{x}\mathbf{z}$, passing through the point $(0,0,h)$
and touching the unit sphere at the point with a positive $\mathbf{x}$-coordinate. Then the tangency point is
$$p:=\left(\,\frac{\sqrt{h^2-1}}{h},0,\frac{1}{h}\,\right)\ .$$ Let $\ell_1$ be the line, tangent to the unit sphere at the point $p$ and forming the angle
$\delta$ with the line $\iota$. Let $\ell_2$ be the line obtained from $\ell_1$ by a rotation around the $\mathbf{z}$-axis by the
angle $2\pi/\mathsf{k}$ where $\mathsf{k}$ is the number of faces of $\mathcal{P}$ sharing a common vertex. The distance to be explored is the distance between lines $\ell_1$ and $\ell_2$. The tangency points of $\ell_1$ and $\ell_2$ are at the same latitude
$$\varphi=\arccos \left(\frac{\sqrt{h^2-1}}{h}\right)\ .$$
Therefore the square of the distance between the lines, obtained from the lines $\ell_1$ and $\ell_2$ in the $\delta$-rotation process, given by the formula (28) from Section 6.1 in \cite{OS}, is
\begin{equation}\label{diab}d^{2}=\frac{4\sin(\alpha)^{2}(1-S^{2})^{2}T^{2}}{(S^{2}+T^{2}
)(1-\sin(\alpha)^{2}S^{2}+\cos(\alpha)^{2}T^{2})}\ .\end{equation}
Here $S=\sin (\varphi)=1/h$, $\alpha=  \pi/\mathsf{k}$ and $T=\tan (\delta)$.

\vskip .2cm
Let $r_m$ be the midraius (the radius of the sphere which touches the middle of each edge of $\mathcal{P}$) and $r_c$ the radius of a circumscribed sphere (the one that passes through all vertices of $\mathcal{P}$).
We want the sphere touching the middles of edges of $\mathcal{P}$ to be unit. Then $$h=r_c/r_m\ .$$

\vskip .2cm
The well-known values of $r_c$ and $r_m$ for the Platonic bodies of edge length $a$ are collected in the following
table; the last column gives the value of $h$.

\begin{center}
\begin{tabular}{c| c| c| c}
& $r_c$ & $r_m$ & $h$ \\[1em]
\hline
&&&\\[-.8em]
tetrahedron& $\displaystyle{\frac{\sqrt{3}}{\sqrt{8}}\,a}$ & $\displaystyle{ \frac{1}{\sqrt{8}}\, a }$& $\sqrt{3}$ \\[1em]
\hline
&&&\\[-.8em]
octahedron& $\displaystyle{ \frac{1}{\sqrt{2}}\, a }$& $\displaystyle{ \frac{1}{2}\, a }$ & $\sqrt{2}$ \\[1em]
\hline
&&&\\[-.8em]
cube& $\displaystyle{ \frac{\sqrt{3}}{2}\, a}$ & $\displaystyle{ \frac{1}{\sqrt{2}}\, a }$ &$\displaystyle{ \frac{\sqrt{3}}{\sqrt{2}} }$ \\[1em]
\hline
&&&\\[-.8em]
icosahedron& $\displaystyle{ \frac{5^{1/4}\,\tau^{1/2}}{2}\, a }$& $\displaystyle{ \frac{\tau}{2}\, a }$ & $\displaystyle{ \frac{5^{1/4}}{\tau^{1/2}} }$\\[1em]
\hline
&&&\\[-.8em]
dodecahedron& $\displaystyle{ \frac{\sqrt{3}\,\tau}{2}\, a }$ & $\displaystyle{ \frac{\tau^2}{2}\, a }$ & $\displaystyle{ \frac{\sqrt{3}}{\tau} }$
\end{tabular}
\end{center}

\paragraph{Remark.} \emph{The construction of this section associates (see denominators of the formulas (\ref{forTT}), (\ref{valdodico}) and (\ref{valdodico2}))  pairs of numbers to each exceptional finite subgroup of $SO(3)$ (the proper symmetry group of either of dual Platonic bodies), or, via the McKay correspondence, to exceptional Lie groups $E_6$, $E_7$ and $E_8$. These pairs are:}
\begin{itemize}
\item[]  \emph{2 and -2 for $E_6$;}

\item[]  \emph{3 and -5 for $E_7$;}

\item[]  \emph{two golden integers, of norms 11 and -19, for $E_8$.}
\end{itemize}
\noindent  \emph{It would be interesting to have a uniform formula giving these values in terms of some Lie groups data.}

\section{Platonic compounds: global picture\vspace{.2cm}}\label{placo}
We will analyze different pairwise distances between the tangent lines in the $\delta$-rotation process. This question is equivalent to the following one. Let $\mathcal{P}$ be a Platonic body. 
The proper symmetry group $G_\mathcal{P}$ acts on the set of unordered pairs of distinct edges of $\mathcal{P}$. The set of pairwise distances is described by the orbits of this action. 

\vskip .2cm
The further analysis of the $\delta$-rotation process for the pairs $\mathcal{O}$ and $\mathcal{I}$ is found in the next two subsections.

\subsection{Octahedron/Cube\vspace{.1cm}}\label{glopiOC}
We recall that the octahedron corresponds to $\delta=0$ and the cube corresponds to $\delta=\pi/2$; during the $\delta$-rotation process the edges 
are rotated counterclockwise.

\vskip .2cm
We fix an edge of the cube; it is shown in purple on Figure \ref{Distances, OC}. The set of orbits of the action of the group $G_\mathcal{O}$ 
on the set of unordered pairs of distinct edges has cardinality five, shown in red, green, blue, black and yellow on Figure \ref{Distances, OC}.
\begin{figure}[H]
\centering
\includegraphics[scale=0.2]{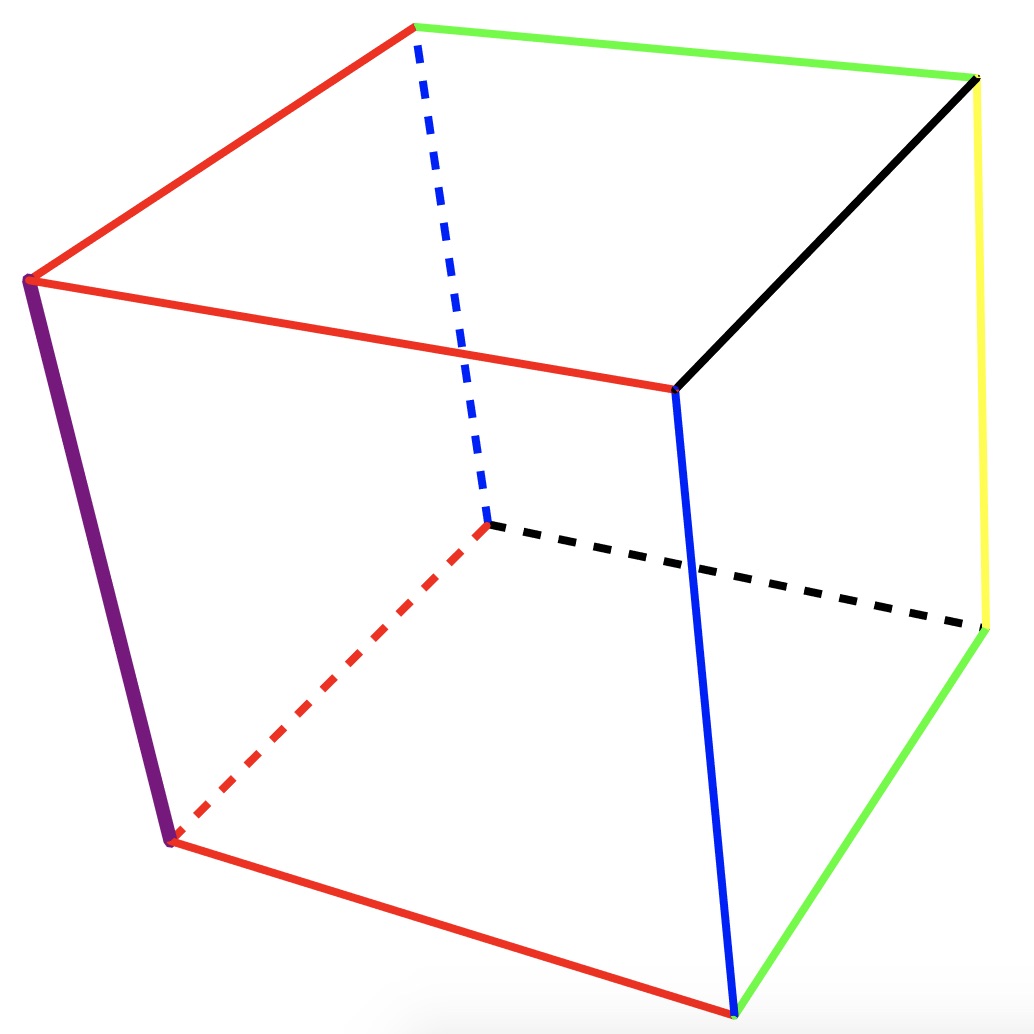}
\vskip .2cm
\caption{Pairwise distances for the pair $\mathcal{O}$}
\label{Distances, OC}
\end{figure}
The graphs of the squares of four distances during the $\delta$-rotation process are shown on Figure \ref{FiveDistancesOC}; $\delta$ varies from $0$ to $\pi/2$.
The colours of the graphs correspond to the colours of the rotated red, green, blue and black edges. 
The remaining fifth distance is the distance between the purple and the yellow edges; it keeps the constant value 2 in the $\delta$-rotation process.

\vskip .2cm
The graphs on Figure \ref{FiveDistancesOC} show that only two out of five distances are relevant for the minimum; these are the graphs shown in red and green on Figure \ref{FiveDistancesOC}. The red graph is the graph of the function (\ref{valdodico}). The green graph is the graph of the function
$$ \frac{ \left(4 \cos (2 \delta) + \sqrt{2} \sin (2 \delta)\right)^2}{6 \cos (\delta)^4 + 8 \sqrt{2} \cos (\delta)^3 \sin (\delta) + 8 \sin (\delta)^4}\ .$$
\begin{figure}[H]
\vspace{-.4cm} 
\centering
$\!\!\!\!\!$\includegraphics[scale=0.47]{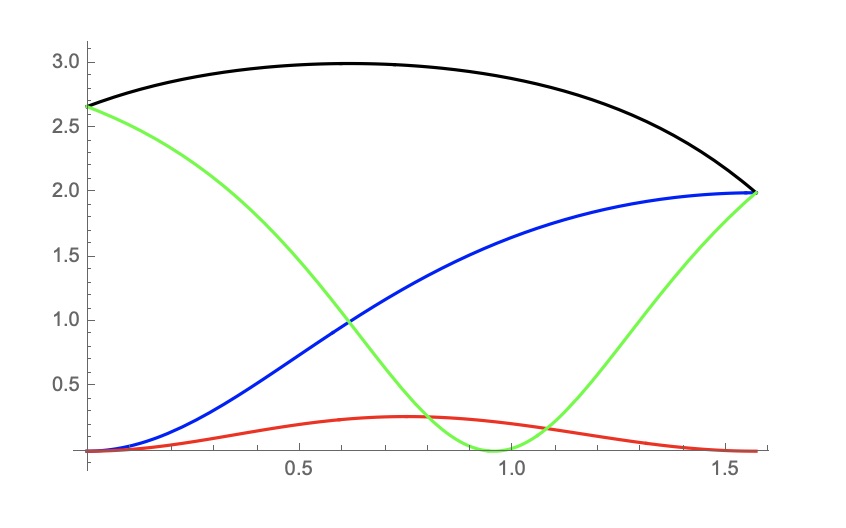}
\vskip -.4cm
\caption{Graphs of squares of four distances for the pair $\mathcal{O}$}
\label{FiveDistancesOC}
\end{figure}
The graph of the minimum of the squares of the distances in the $\delta$-rotation process is shown on Figure \ref{MinDistOC}. The maximum of the graph corresponds to the value $\delta_\mathcal{O}$ given by (\ref{valdodico0}). 
\begin{figure}[H]
\vspace{.2cm} \centering
$\!\!\!\!\!\!\!$\includegraphics[scale=0.51]{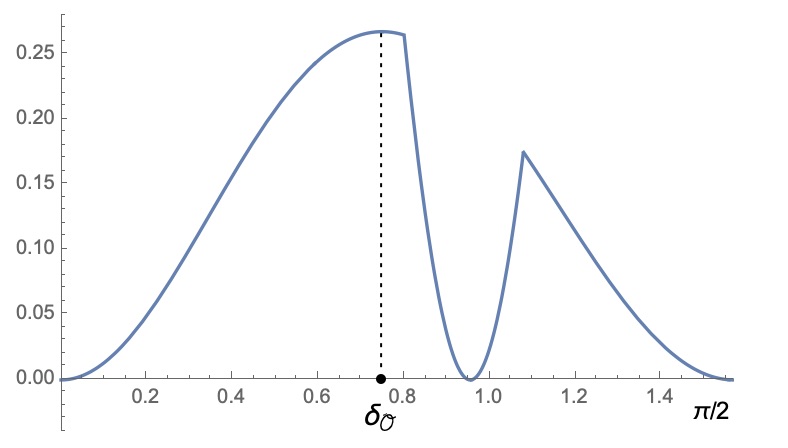}
\vskip -.2cm
\caption{Graph of the square of the minimal distance for the pair $\mathcal{O}$}
\label{MinDistOC}
\end{figure}
The first corner point on Figure \ref{MinDistOC} occurs at $$\delta\simeq 0.800811\ .$$ The value of the square of the minimal distance at this point is $$\simeq 0.26534\ ,$$ compare to the value (\ref{valdodicob}).

\subsection{Icosahedron/Dodecahedron\vspace{.1cm}}\label{glpicID}
In the $\delta$-rotation process we let the icosahedron correspond to $\delta=0$ and the dodecahedron correspond to $\delta=\pi/2$; during the $\delta$-rotation process the edges 
are rotated counterclockwise.

\vskip .2cm
The set of orbits of the action of the group $G_\mathcal{I}$ on the set of unordered pairs of distinct edges has cardinality eleven. We illustrate the orbits on Figure \ref{Distances, ID}. We fix an edge of the dodecahedron. 
It is marked by a red disk on the dodecahedron net on Figure \ref{Distances, ID}. The eleven different distances during the $\delta$-rotation process 
correspond to the edges on the dodecahedron net with the marks from 1 to 11. 
\begin{figure}[H]
\vspace{.2cm} \centering
\includegraphics[scale=0.4]{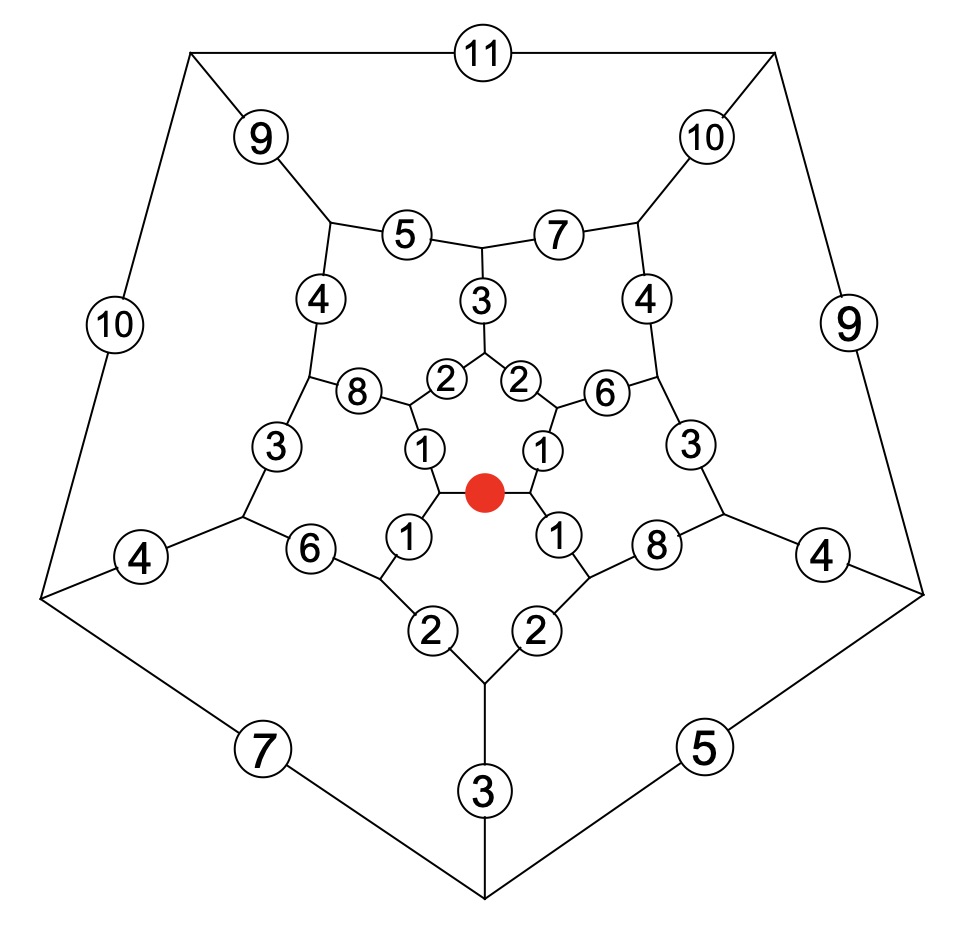}
\vskip -.2cm
\caption{Pairwise distances for the pair $\mathcal{I}$}
\label{Distances, ID}
\end{figure}

\vskip .2cm
The graphs of the ten distances during the $\delta$-rotation process, are shown on Figure \ref{TenDistances, ID}; $\delta$ varies from $0$ to $\pi/2$. 
The remaining eleventh distance is between opposite edges; it keeps the constant value 2 during the $\delta$-rotation process.

\vskip .2cm
These graphs on Figure \ref{TenDistances, ID} show that only four out of eleven distances are relevant for the minimum. 
\begin{figure}[H]
\vspace{-.4cm} \centering
\includegraphics[scale=0.42]{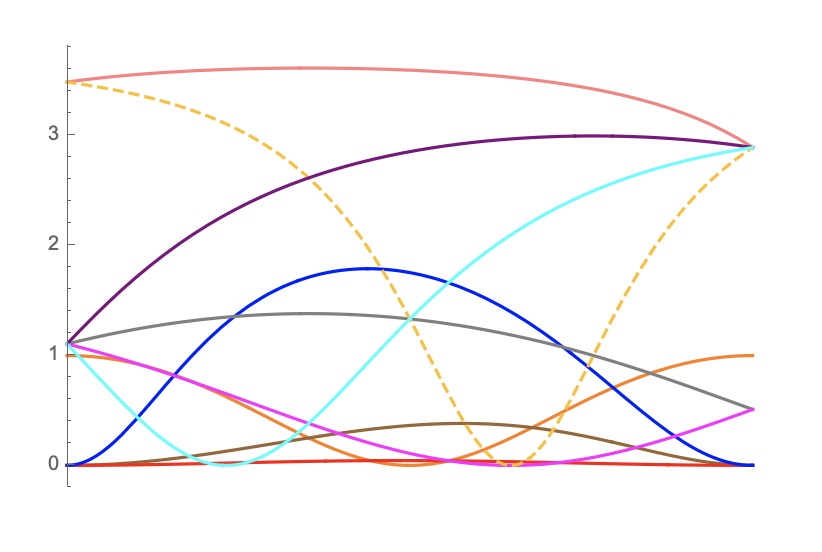}
\vskip -.3cm
\caption{Graphs of squares of ten distances for the pair $\mathcal{I}$}
\label{TenDistances, ID}
\end{figure}
\vskip -.2cm
The graphs of relevant distances are shown on Figure \ref{RelevantDistancesID}.
\begin{figure}[H]
\centering
\includegraphics[scale=0.42]{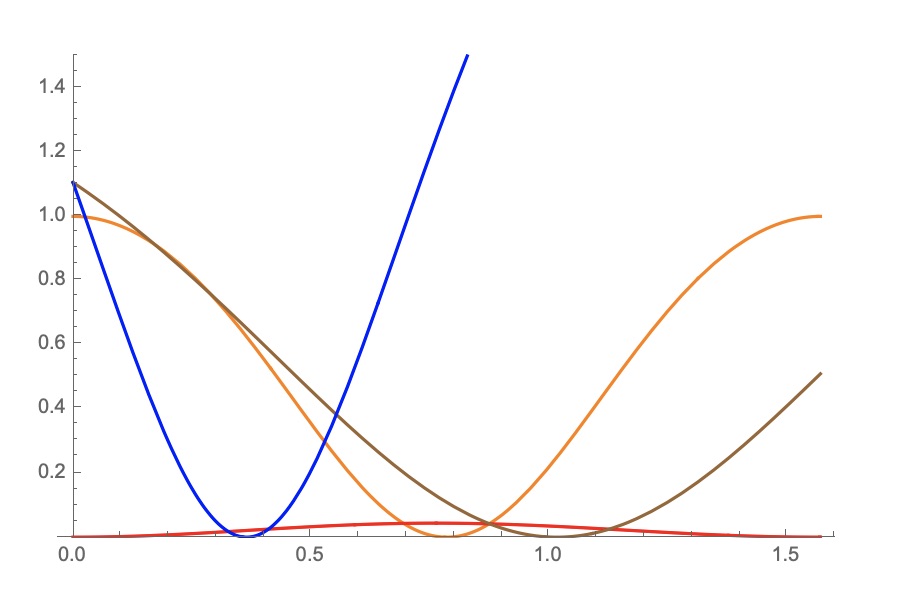}
\vskip -.2cm
\caption{Relevant distances for the pair $\mathcal{I}$}
\label{RelevantDistancesID}
\end{figure}

We have explicitly calculated the distances during the $\delta$-rotation process with the help of Mathematica \cite{W}.

\vskip .2cm
The red graph corresponds to the distance marked by 1 on Figure \ref{Distances, ID}. This function is given by the formula (\ref{valdodico2}). 

\vskip .2cm
The blue graph corresponds to the distance marked by 5 on Figure \ref{Distances, ID}. This function is
\begin{equation}\label{secora0}\frac{8\left( \sqrt{5}\sin (2\delta)-2\cos (2\delta)\right)^2}{21+4\sqrt{5}+4\sqrt{5}\cos (2\delta)-\cos (4\delta)+8\sin (2\delta)-4\sqrt{5}\sin (4\delta) }\ .\end{equation}

The orange line corresponds to the distance marked by 3 on Figure \ref{Distances, ID}. This function is
\begin{equation}\label{secora}\frac{4\cos^2 (2\delta)}{3+\cos^2 (2\delta)}\ .\end{equation}

The brown graph corresponds to the distance marked by 8 on Figure \ref{Distances, ID}. This function is
\begin{equation}\label{thibr1}\frac{8(2\cos (2\delta)+\sin (2\delta))^2}{25+8\sqrt{5}+4\tau^3(2\cos (2\delta)-\sin (2\delta))+3\cos (4\delta)+4\sin (4\delta) }\ .
\end{equation}

The distance marked by 9 is also important in the analysis. Its graph is dashed on Figure \ref{TenDistances, ID}. This function is 
\begin{equation}\label{thibr2}\frac{8\tau (2\cos (2\delta)+\sin (2\delta))^2}{\tau^3 (25-8\sqrt{5})+4(2\cos (2\delta)-\sin (2\delta))+\tau^3(3\cos (4\delta)+4\sin (4\delta)) }\ .\end{equation}
The graph of this function goes above 
the graph of the $min$ function (the minimum of the squares of the distances) 
everywhere except  the point  
$$\frac{\pi-\arctan (2)}{2}\ .$$
At this point the functions (\ref{thibr1}) and (\ref{thibr2}) vanish simultaneously and this complicates the corresponding minimal configuration, see 
details in Subsection \ref{minimaID}.

\vskip .2cm
The graph of the minimum of the squares of the distances in the $\delta$-rotation process is shown on Figure \ref{PlotMinID}. At $\delta_\mathcal{I}$ ($\delta_\mathcal{I}$ is slightly greater than $\delta_{\mathrm{max}}$) the $min$ function is given by (\ref{secora}) while the value (\ref{valdodico3}) is the value of the function (\ref{valdodico2}). For this reason the analysis of the pair $\mathcal{I}$ has to be continued.  
\begin{figure}[H]
\centering
\includegraphics[scale=0.3]{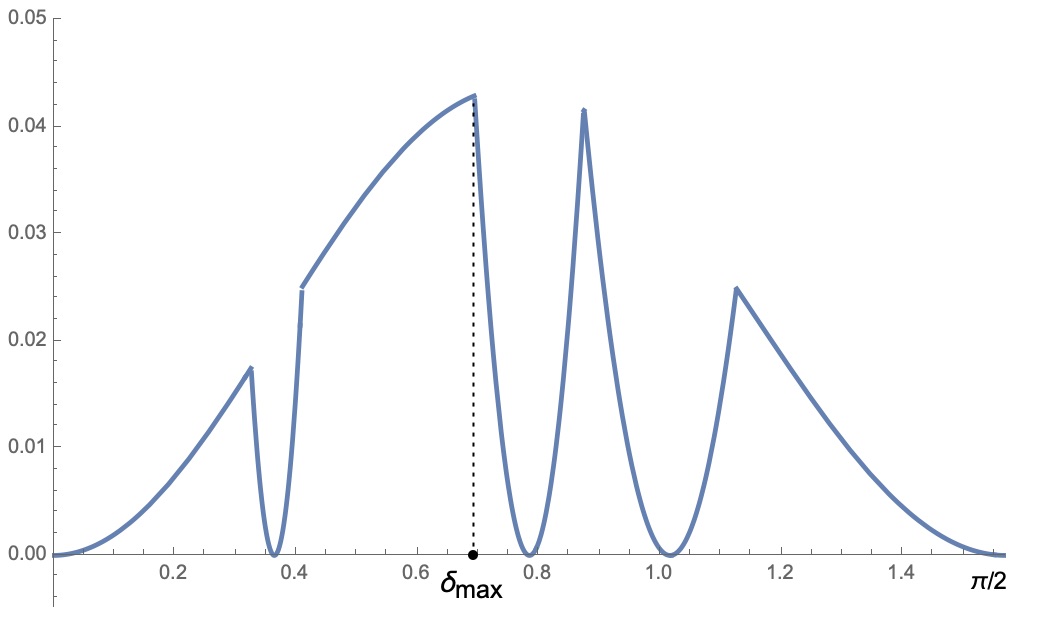}
\vskip .2cm
\caption{Graph of the square of the minimal distance for the pair $\mathcal{I}$}
\label{PlotMinID}
\end{figure}
The seemingly peak, for $\delta$ between $0.874$ and $0.876$, on Figure \ref{PlotMinID}, is illusive. The fine structure of the graph 
around these values of $\delta$ is shown on Figure \ref{FineStructure}.
\begin{figure}[H]
\centering
\includegraphics[scale=0.4]{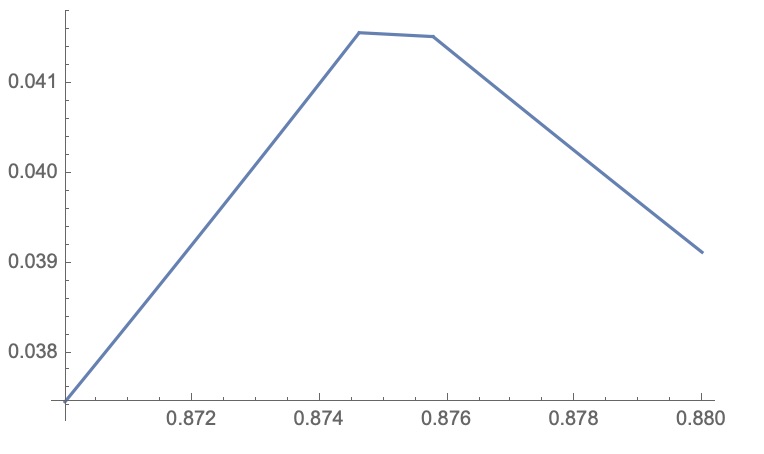}
\vskip .2cm
\caption{Fine structure}
\label{FineStructure}
\end{figure}

\vskip .2cm
The resulting maximal compound, at the angle $\delta_{\text{max}}$, of thirty cylinders is shown on Figures \ref{MaxID-ConfigViewFrom5-axis} and \ref{MaxID-ConfigViewFrom3-axis}.
\newpage
\begin{figure}[H]
\vspace{-3cm} 
\centering
\includegraphics[scale=0.34]{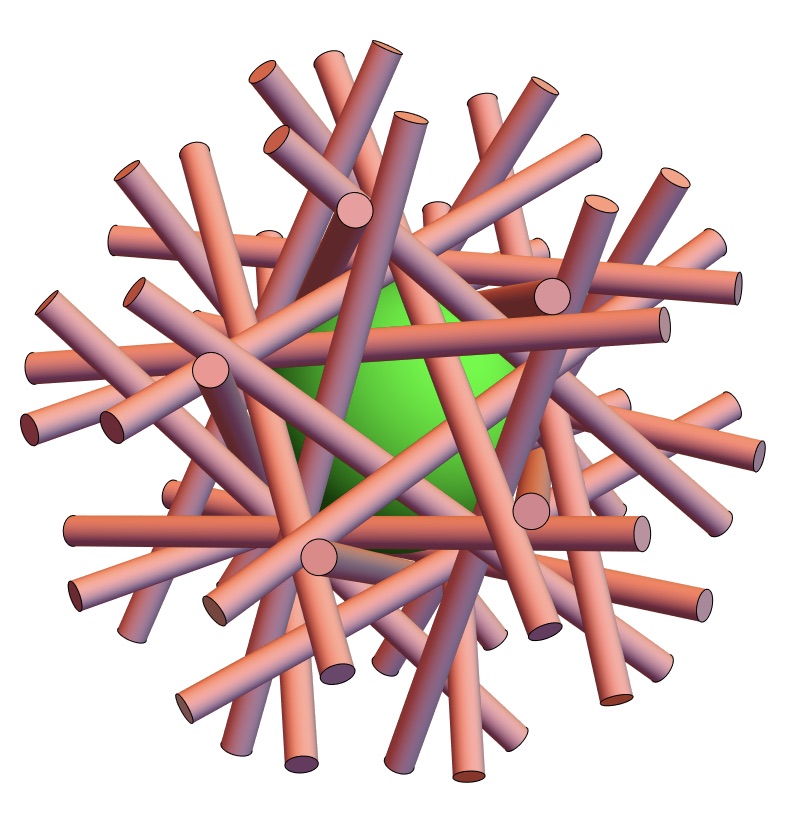}
\caption{Maximal configuration, view from the tip of a 5-fold axis}
\label{MaxID-ConfigViewFrom5-axis}
\end{figure}
\begin{figure}[H]
\vspace{-.4cm} 
\centering
\includegraphics[scale=0.36]{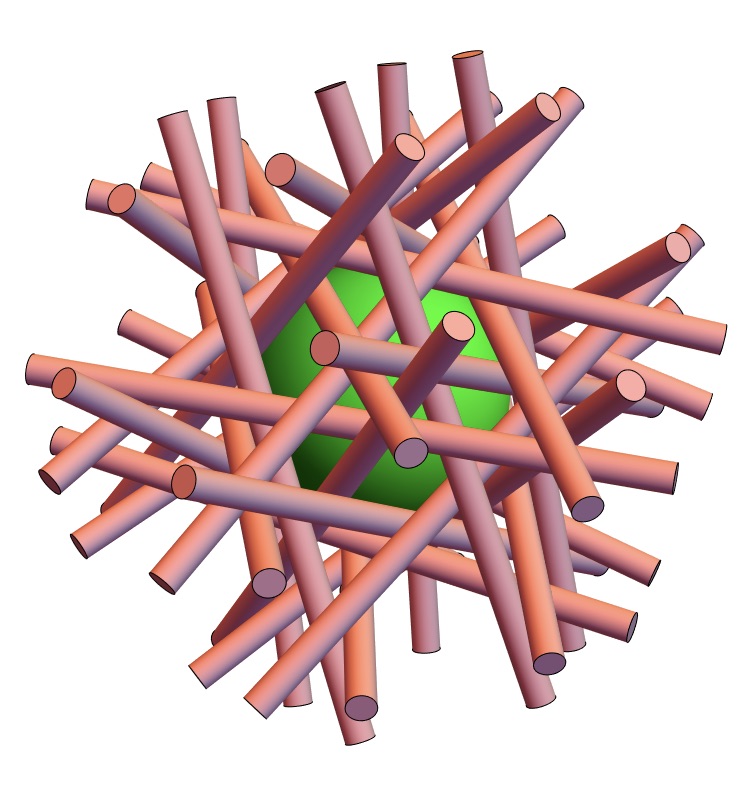}
\caption{Maximal configuration, view from the tip of a 3-fold axis}
\label{MaxID-ConfigViewFrom3-axis}
\end{figure}

We summarize the results about the maximal distance in the $\delta$-rotation process for the pair $\mathcal{I}$.

\begin{theorem} The maximum distance in the $\delta$-rotation process for the pair $\mathcal{I}$ is achieved at the angle 
\begin{equation}\label{valdodicoI}\delta_{\text{max}}=\arctan\left(\sqrt{t_0}\right)\ ,\end{equation}
where $$t_0\simeq 0.694356$$ is a root of the polynomial
$$5t^6-80t^5+190t^3-4t^2-84t+9\ ;$$ 
all roots of this polynomial are real.
Approximately, 
$$\delta_{\text{max}}\simeq 0.694707\ .$$

The value of the square of the minimal distance 
at $\delta_{\text{max}}$ is approximately 
$$d_{\text{max}}\simeq 0.0429216\ ,$$ 
compare with (\ref{valdodico3}). The corresponding radius of cylinders is approximately $$r_{\text{max}}\simeq0.115558\ ,$$ compare with (\ref{valdodico4}).
\end{theorem}

\section{Minima\vspace{.2cm}}\label{mincom}
For completeness, we describe here the critical points corresponding to the local minima during the $\delta$-rotation process.
There are two `trivial' local minima corresponding to the angles $\delta=0$ and $\delta=\pi/2$; these are Platonic bodies themselves. Other local minima happen when continuations of certain non-neighboring edges start to intersect. 

\vskip .2cm
For each minimum we present a picture of intersecting tangent lines. On the plane projection the line which goes below looses 
its color at the intersection, so the color of the intersection corresponds to the line which goes above. 

\subsection{Octahedron/Cube\vspace{.1cm}}
For the pair $\mathcal{O}$ we have only one non-trivial minimum, see Figure \ref{MinDistOC}.
The minimum happens at the angle $$\delta=\arctan \left(\sqrt{2}\right)\ .$$ 

At this angle every tangent line $\ell$ starts to intersect with two other tangent lines, see Figure \ref{Distances, OC}. Therefore the twelve edges split into four triangles.

\vskip .2cm
The resulting figure, shown on Figure \ref{Octahedron-CubeMinimum}, is formed by four triangles whose edge length is $2\sqrt{3}$. 
The triangles are pairwise simply linked. 

\vskip .2cm
One can describe this figure as follows. Choose a three-fold axis of the proper symmetry group $G_{\mathcal{O}}$ of the
octahedron/cube. Draw an equilateral triangle $\mathsf{T}$ in the plane orthogonal to the axis and passing through the origin. Then the figure is obtained by application of the rotations from $G_{\mathcal{O}}$  to the triangle $\mathsf{T}$.  
\begin{figure}[H]
\vspace{-.4cm} 
\centering
\includegraphics[scale=0.44]{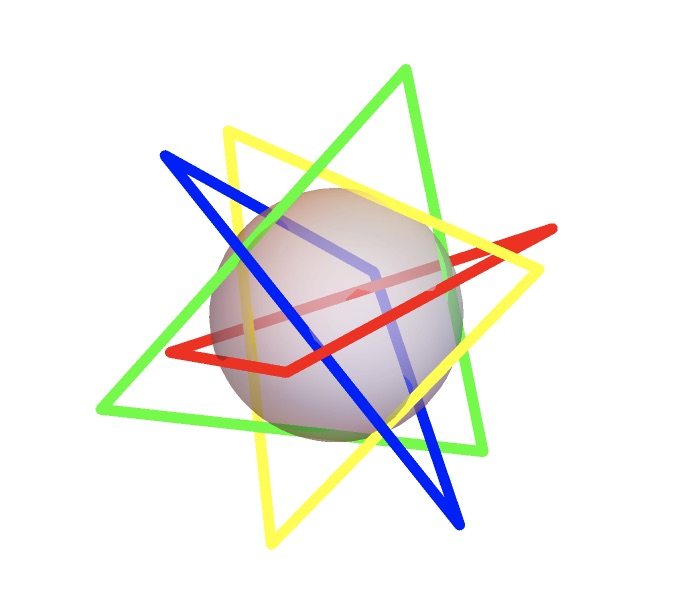}
\vskip -.2cm
\caption{Octahedron/cube minimum}
\label{Octahedron-CubeMinimum}
\end{figure}
The convex hull of the four triangles is the cuboctahedron, of edge length 2, shown on Figure \ref{Octahedron-CubeMinimumConvexHull}. 
\begin{figure}[H]
\centering
\includegraphics[scale=0.42]{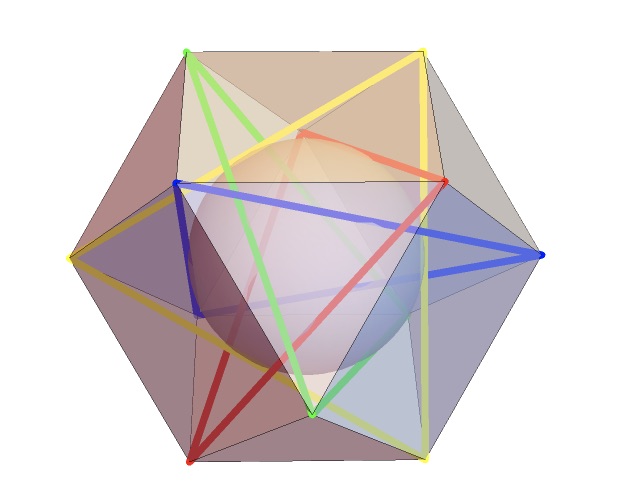}
\caption{Octahedron/cube minimum, convex hull}
\label{Octahedron-CubeMinimumConvexHull}
\end{figure}

\subsection{Icosahedron/Dodecahedron\vspace{.1cm}}\label{minimaID}
As it can be seen on Figure \ref{PlotMinID}, there are three non-trivial values of $\delta$ for which certain tangent lines start to intersect. 

\paragraph{1.} The first minimum happens at the angle $$\delta=\frac{1}{2}\arctan \left(\frac{2}{\sqrt{5}}\right)\ .$$ 

The intersections correspond to distances marked by 5 on Figure \ref{Distances, ID}. Since for every tangent line $\ell$ there are two other tangent lines at the distance marked by 5 from $\ell$, the thirty edges have to split into ten triangles.

\vskip .2cm
The resulting figure is a union of ten triangles
of edge length $2\sqrt{3}$. It is shown on Figure \ref{FirstMinimum}. The triangles are pairwise simply linked. 

\vskip .2cm
One can describe this figure as follows. Choose a three-fold axis of the proper symmetry group $G_{\mathcal{I}}$  of the 
icosahedron/dodecahedron. Draw an equilateral triangle $\mathsf{T}$ in the plane orthogonal to the axis and passing through the origin. Then the figure is obtained by application of the rotations from $G_{\mathcal{I}}$  to the triangle $\mathsf{T}$.  

\begin{figure}[H]
\vspace{-.4cm} \centering
\includegraphics[scale=0.48]{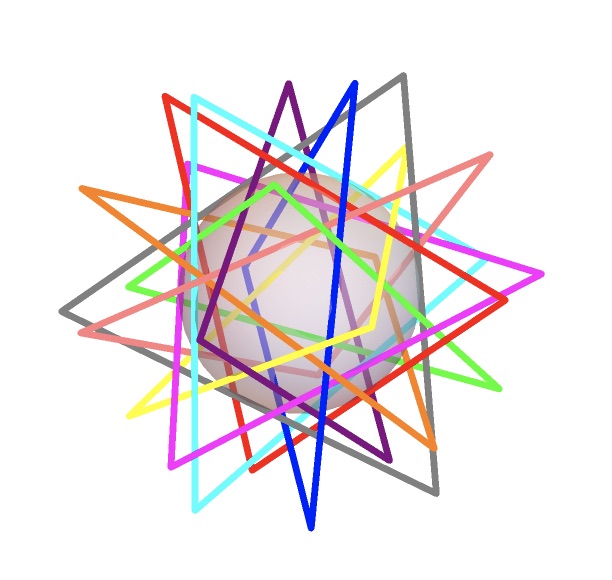}
\vskip .2cm
\caption{First minimum}
\label{FirstMinimum}
\end{figure}
The convex hull of the figure is the icosidodecahedron, of edge length $\sqrt{5}-1$, shown on Figure \ref{fig-Icosidodecahedron}.
\begin{figure}[H]
\vskip -.3cm\centering
\includegraphics[scale=0.3]{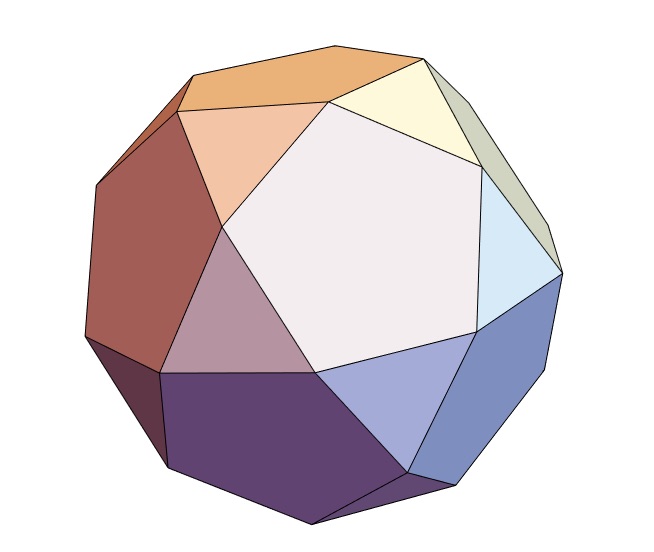}
\vskip .2cm
\caption{Icosidodecahedron}
\label{fig-Icosidodecahedron}
\end{figure}

\paragraph{2.} The second minimum happens at the angle $$\delta=\frac{\pi}{4}\ .$$ 

\vskip .2cm
Here the intersections correspond to distances marked by 3 on Figure \ref{Distances, ID}. Since each tangent line enters into four pairs of type marked by 3, 
the thirty edges have to split into five one-skeletons of the tetrahedron.

\vskip .2cm
The resulting figure is the compound of one-skeletons, of edge length $2\sqrt{2}$, of the five tetrahedra inscribed in the dodecahedron. It is shown on Figure \ref{SecondMinimum}.

\begin{figure}[H]
\vspace{.2cm} \centering
\includegraphics[scale=0.76]{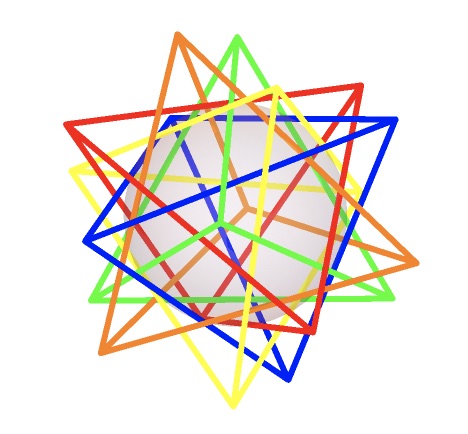}
\vskip .2cm
\caption{Second minimum}
\label{SecondMinimum}
\end{figure}

\vskip .4cm
The pairwise linking of the one-skeletons is shown on Figure \ref{LinkingTetrahedra}. 

\begin{figure}[H]
\centering
\includegraphics[scale=0.52]{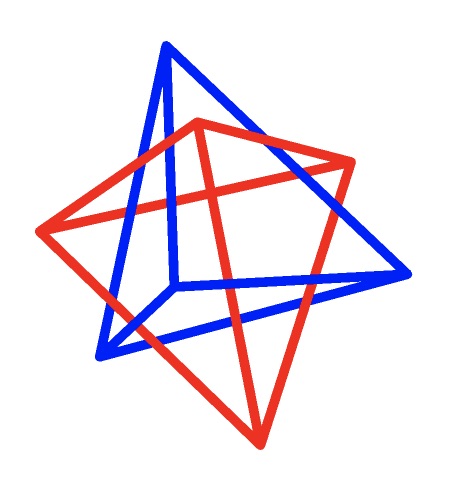}
\vskip .2cm
\caption{Linking of tetrahedra}
\label{LinkingTetrahedra}
\end{figure}

\paragraph{3.} The third minimum happens at the angle \begin{equation}\label{thmin}\delta=\arctan (\tau)\equiv\frac{\pi-\arctan (2)}{2}\ ,\end{equation} 
where $\tau$ is the golden ratio. 

\vskip .2cm
At this value of $\delta$ two distances (\ref{thibr1}) and (\ref{thibr2}) vanish simultaneously. These are distances marked by 8 and 9 on Figure \ref{Distances, ID}.
For each tangent line $\ell$ there are two other lines of type marked by 8 and two more lines marked by 9. So each tangent line intersects four other lines, and the intersections are of two different combinatoric types.

\vskip .2cm
It turns out that at the angle (\ref{thmin}) the edges form pentagonal stars, so the resulting figure is the compound of six pentagonal stars. It is shown on Figure \ref{ThirdMinimum}. 
\begin{figure}[H]
\vspace{-.2cm} \centering
$\!\!\!\!$\includegraphics[scale=0.71]{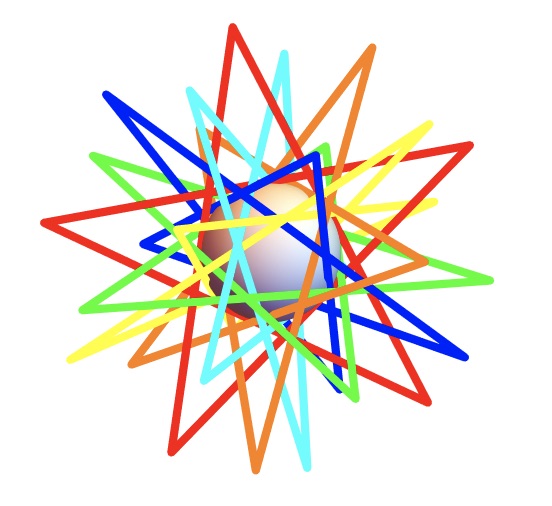}
\vskip .2cm
\caption{Third minimum}
\label{ThirdMinimum}
\end{figure}
One can describe this figure as follows. Choose a five-fold axis of the proper symmetry group $G_{\mathcal{I}}$  of the 
icosahedron/dodecahedron. Draw a regular pentagonal star $\mathsf{P}$ in the plane orthogonal to the axis and passing through the origin. Then the figure is obtained by application of the rotations from $G_{\mathcal{I}}$  to the star $\mathsf{P}$.   

\vskip .3cm
The inner pentagons of the stars form a compound shown on Figure \ref{ThirdMinimumPentagons}. The pentagons are pairwise simply linked.
\begin{figure}[H]
\vspace{.2cm} \centering
\includegraphics[scale=0.72]{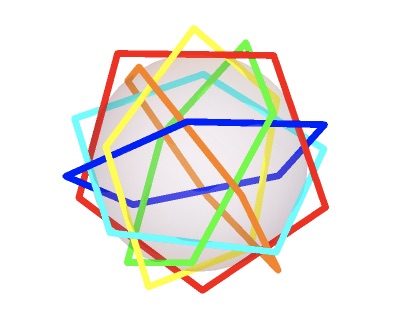}
\vskip .2cm
\caption{Third minimum, inner pentagons}
\label{ThirdMinimumPentagons}
\end{figure}
The convex hull of either ten stars or ten pentagons is again an icosidodecahedron. 

\section{Concluding remarks\vspace{.1cm}}
For the pair $\mathcal{T}$, the square of the minimal distance in the $\delta$-rotation process is a smooth function with a single local maximum.  
Theorem 1 in \cite{OS-O6} asserts that this maximum is a local maximum in the space of all possible configurations of six tangent lines
mod $SO(3)$.

\vskip .2cm
We present several conjectures concerning two other dual pairs. Let $\mathcal{Z}_k$ denote the configuration space, mod $SO(3)$, of $k$ non-overlapping cylinders touching the unit ball.
Let $\mathcal{Z}_k(\geq\mathsf{R})$ denote the subspace of $\mathcal{Z}_k$ consisting of cylinders
of common radius greater than or equal to $\mathsf{R}$.
Let $\mathcal{U}$ be a dual pair of Platonic bodies, either $\mathcal{O}$ or $\mathcal{I}$. We denote by $\mathsf{R}_\mathcal{U}(\delta)$ the radius of touching cylinders corresponding to 
the minimal distance in the $\delta$-rotation process related to the pair $\mathcal{U}$.

\paragraph{Remark.}  \emph{It is interesting to note that the number of pairwise distances is equal to $3{\mathsf{t}}-1$ where $\mathsf{t}$ takes values 1,2 and 4 for, 
respectively $\mathcal{T}$, $\mathcal{O}$ and $\mathcal{I}$. The minimum of the squares of the distances is the minimum of exactly $\mathsf{t}$ smooth functions.
The number of local maxima of the minimum of the squares of the distances is also equal to $\mathsf{t}$.}

\vskip .4cm
We believe that our conjectures can be proved (or disproved) using the techniques of our proofs of the local maximality of the configuration 
$C_{\mathfrak{m}}$ in \cite{OS-C6} and of the configuration $O_6$ in \cite{OS-O6}.

\subsubsection*{Pair octahedron/cube}
Let $\delta_{\mathcal{O},1}$ and $\delta_{\mathcal{O},2}$ be the values of the  angle $\delta$ corresponding to two local maxima of the function  
$\mathsf{R}_\mathcal{O}(\delta)$, see Figure \ref{MinDistOC}. Here  $\delta_{\mathcal{O},1}=\delta_{\mathcal{O}}$. In the $\delta$-rotation process  
for the pair $\mathcal{O}$, there are three distinguished values of $\delta$ for $0<\delta<\pi_2$: 
$\delta_{\mathcal{O},1}$, $\delta_{\mathcal{O},2}$ and one more corner point.

\begin{itemize}
\item {\bf Conjecture i$_\mathcal{O}$, weak form.} The configurations of twelve cylinders of radius $\mathsf{R}_\mathcal{O}(\delta_{\mathcal{O},j})$, $j=1,2$, 
are local maxima in the space $\!\mathcal{Z}_{12}$.
\item {\bf Conjecture i$_\mathcal{O}$, strong form.} The configurations of twelve cylinders of radius $\mathsf{R}_\mathcal{O}(\delta_{\mathcal{O},j})$, $j=1,2$, 
are, moreover, isolated points in the space $\mathcal{Z}_{12}(\geq\mathsf{\mathsf{R}_\mathcal{O}(\delta_{\mathcal{O},j})})$. 
\end{itemize}
\paragraph{Comment about the remaining corner point.} Here the left and right derivatives of the minimum of the squares of the distances have the same sign. 
Therefore the resulting (at this angle) configuration of twelve cylinders of the corresponding radius 
is not a critical point in the space $\mathcal{Z}_{12}$. 

\subsubsection*{Pair icosahedron/dodecahedron}
Let $\delta_{\mathcal{I},j}$, $j=1,2,3,4$, be the values of the  angle $\delta$ corresponding to four local maxima of the function $\mathsf{R}_\mathcal{I}(\delta)$, see Figures \ref{PlotMinID} and \ref{FineStructure}. Besides, there are two corner points of the minimum of the squares of the distances. Altogether, there are six distinguished values 
of the angle $\delta$ in the $\delta$-rotation process for the pair $\mathcal{I}$.
As it is seen on Figures \ref{PlotMinID} and \ref{FineStructure}, the left and right derivatives of the squares of the distances have the same sign for each corner point which is not a local maximum. 

\begin{itemize}
\item {\bf Conjecture i$_\mathcal{I}$, weak form.} The configurations of thirty cylinders of radius $\mathsf{R}_\mathcal{I}(\delta_{\mathcal{I},j})$, $j=1,2,3,4$, 
are local maxima in the space $\!\mathcal{Z}_{30}$.  
\item {\bf Conjecture i$_\mathcal{I}$, strong form.} The configurations of thirty cylinders of radius $\mathsf{R}_\mathcal{I}(\delta_{\mathcal{I},j})$, $j=1,2,3,4$, 
are, moreover, isolated points in the space $\mathcal{Z}_{30}(\geq\mathsf{\mathsf{R}_\mathcal{I}(\delta_{\mathcal{I},j})})$. 
\end{itemize}

\subsubsection*{Minima}
The $\delta$-rotation process for the pair $\mathcal{O}$ exhibits the unique non-trivial minimum, see Figure \ref{Octahedron-CubeMinimum}. 
The convex span of the vertices on Figure \ref{Octahedron-CubeMinimum} is the cuboctahedron, Figure \ref{Octahedron-CubeMinimumConvexHull}. 
Its cousin for the pair $\mathcal{I}$ is the  icosidodecahedron which appears as the convex span of the vertices of the compounds of rotated edges for the two local minima in the $\delta$-rotation process for the pair $\mathcal{I}$. 

\vskip .2cm
The vertices of the cuboctahedron form the FCC configuration of twelve 
balls touching the unit central ball. Thus the configuration of the thirty balls with centres in the vertices of the icosidodecahedron is a natural analogue of the FCC configuration. 

\vskip .2cm
The balls, touching the central unit ball and centered in the vertices of the icosidodecahedron, touch each other when their common radius is equal to
$$\frac{1}{\sqrt{5}}\ .$$  
Let us call $\mathcal{ID}$ this configuration of balls. By analogy with the FCC configuration, we expect that the configuration $\mathcal{ID}$ is critical and 
is unlockable by a move similar to the Coxeter move, see \S$\,$ 8.4 in \cite{C}, for the configuration FCC.

\vskip .6cm\noindent
{\textbf{Acknowledgements.}
Part of the work of S. S. has been carried out in the framework
of the Labex Archimede (ANR-11-LABX-0033) and of the A*MIDEX project (ANR-11-
IDEX-0001-02), funded by the Investissements d'Avenir French Government program
managed by the French National Research Agency (ANR). Part of the work of S. S. has
been carried out at IITP RAS. The support of Russian Foundation for Sciences (project
No. 14-50-00150) is gratefully acknowledged by S. S. The work of O. O. was supported by
the Program of Competitive Growth of Kazan Federal University and by the grant RFBR
17-01-00585.}

\end{document}